\documentclass{article}

\newcommand{\Addresses}{{
		\bigskip
		\footnotesize
		
		\textsc{Mathematical Institute, University of Oxford, Oxford, OX2 6GG, UK}\par\nopagebreak
		\textit{E-mail address:} \texttt{ofir.goro@gmail.com}		
		
		\bigskip
  \textsc{D{\'e}partment  de Math{\'e}matiques et Statistique,   Universit{\'e} de Montr{\'e}al, Montr{\'e}al, QC  H3C 3J7, Canada}
  \par\nopagebreak
		\textit{E-mail address:}
\texttt{valkovaleva42@gmail.com}
}}

\author{Ofir Gorodetsky, Valeriya Kovaleva} 
\title{Equidistribution of high traces of random matrices over finite fields and cancellation in character sums of high conductor}
\date{}

\usepackage[margin=1in]{geometry}
\usepackage{amssymb}
\usepackage{amsfonts}
\usepackage{amsthm}
\usepackage{amsmath}
\usepackage{hyperref}
\usepackage{mathtools}
\usepackage{url}
\usepackage{xcolor}

\theoremstyle{plain}
\newtheorem{thm}{Theorem}[section]
\newtheorem{lem}[thm]{Lemma}  
\newtheorem{prop}[thm]{Proposition}
\newtheorem{cor}[thm]{Corollary}

\theoremstyle{remark}
\newtheorem{remark}{Remark}[section]

\newcommand{\PP}{\mathbb{P}}
\newcommand{\CC}{\mathbb{C}}
\newcommand{\NN}{\mathbb{N}}
\newcommand{\ZZ}{\mathbb{Z}}
\newcommand{\FF}{\mathbb{F}}
\newcommand{\EE}{\mathbb{E}}

\newcommand{\Mnq}{\mathcal{M}_{n,q}}
\newcommand{\Mq}{\mathcal{M}_{q}}
\newcommand{\Pnq}{\mathcal{P}_{n,q}}
\newcommand{\Pq}{\mathcal{P}_{q}}

\newcommand{\Tr}{\mathrm{Tr}}
\newcommand{\glnq}{\mathrm{GL}_n(\mathbb{F}_q)}

\newcommand{\barfq}{\overline{\FF_q}}
\newcommand{\gfnc}{\mathrm{P}_{\mathrm{GL}}}

\newcommand{\ordpa}{\mathrm{ord}_p(a)}
\newcommand{\charfq}{\mathrm{char}(\FF_q)}
\newcommand{\Mqt}{\mathcal{M}^{gl}_q}

\newcommand{\brackets}[1]{\left(#1\right)}

\numberwithin{equation}{section}
\mathtoolsset{showonlyrefs}

\makeatletter
\newcommand{\subjclass}[2][2020]{%
	\let\@oldtitle\@title%
	\gdef\@title{\@oldtitle\footnotetext{#1 \emph{Mathematics subject classification:} #2}}%
}
\makeatother


\begin{document}
\maketitle
\begin{abstract}
Let $g$ be a random matrix distributed according to uniform probability measure on the finite general linear group $\glnq$. We show that  $\Tr(g^k)$ equidistributes on $\FF_q$ as $n \to \infty$ as long as $\log k=o(n^2)$ and that this range is sharp. We also show that nontrivial linear combinations of $\Tr(g^1),\ldots, \Tr(g^k)$ equidistribute as long as $\log k =o(n)$ and this range is sharp as well. 
Previously equidistribution of either a single trace or a linear combination of traces  was only known for $k \le c_q n$, where $c_q$ depends on $q$, due to work of the first author and Rodgers. 

We reduce the problem to exhibiting cancellation in certain short character sums in function fields. For the equidistribution of $\Tr(g^k)$ we end up showing that certain \textit{explicit} character sums modulo $T^{k+1}$ exhibit cancellation when averaged over monic polynomials of degree $n$ in $\FF_q[T]$ as long as $\log k = o(n^2)$. This goes far beyond the classical range $\log k =o(n)$ due to Montgomery and Vaughan. To study these sums we build on the argument of Montgomery and Vaughan but exploit additional symmetry present in the considered sums. 
\end{abstract}

\section{Introduction}

Fix $\FF_q$, the finite field of $q$ elements. We denote its characteristic by $\charfq$. Let $g \in \glnq$ be an invertible $n \times n$ matrix over $\FF_q$ chosen according to the uniform probability measure. 
The first author and Rodgers \cite[Thm.~1.1]{Gorodetsky2021} showed that $\Tr(g^k)$ equidistributes in $\FF_q$ as $n \to \infty$, uniformly for $k \le c_q n$ for some sufficiently small $c_q$ depending on $q$, and the rate of convergence is superexponential. However, nothing beyond $k=O(n)$ was known. We describe our three main results.
\begin{thm}\label{thm:matrixsingle}
Let $g \in \glnq$ be chosen uniformly at random. Let $k = k(n)$ be a positive integer such that $\log k=o(n^2)$. The distribution of $\Tr(g^k)$ tends to the uniform distribution on $\FF_q$ as $n$ tends to $\infty$.
\end{thm}
The range $\log k = o(n^2)$ in Theorem \ref{thm:matrixsingle} is optimal in the sense that we cannot replace it with $\log k = O(n^2)$. Indeed, if we take $k = |\glnq|= \prod_{i=0}^{n-1}(q^n-q^i)$ then $\log k  \asymp_q n^2$ and, by Lagrange's theorem, $g^k = I_n$ for $g \in \glnq$ and so $\Tr(g^k) \equiv n$ does not equidistribute. (This shows $\Tr(g^k)$ is periodic in $k$.) Theorem~\ref{thm:matrixsingle} leaves open the question, not explored in this paper, whether the range can be extended to $\log k \le c_q n^2$ for some constant $c_q>0$ depending on $q$; we believe the answer is negative. We also prove a theorem for combination of traces.
\begin{thm}\label{thm:matrix}
Let $g \in \glnq$ be chosen uniformly at random. Let $k = k(n)$ be a positive integer such that $\charfq \nmid k$ and $\log k=o(n)$. Let $a_j \in \FF_q$ for $j = 1,\ldots, k$ be arbitrary constants with $a_k \neq 0$. Then the distribution of $\sum_{1 \le i \le k}a_i \Tr( g^i)$ tends to the uniform distribution on $\FF_q$ as $n$ tends to $\infty$.
\end{thm}
Over $\FF_q$ we have $\Tr(g^{\charfq}) = \Tr(g)^{\charfq}$ and the condition $\charfq \nmid k$ is necessary to avoid trivial linear combinations.
The range $\log k=o(n)$ in Theorem \ref{thm:matrix} is optimal in the sense that it cannot be replaced by $\log k =O(n)$, as we now demonstrate by giving an example where $ \sum_{1 \le i \le k} a_i \Tr(g^i)$ ($ \charfq \nmid k$, $a_k \neq 0$) is not equidistributed and $\log k \asymp_q n$. Let $\barfq$ denote an algebraic closure of $\FF_q$, and let $f \in \FF_q[T]$ be a monic polynomial. Given its factorization $f(T) = \prod_{j=1}^{\deg f} (T-\lambda_j)$ over $\barfq$, we define its $i$th ($i \ge 0$) power sum symmetric polynomial as
\begin{equation}\label{eq:pk}
p_i(f) := \sum_{j=1}^{\deg f} \lambda_j^i \in \barfq.
\end{equation}
By Newton's identities $p_i(f)$ is an integral multivariate polynomial in coefficients of $f$, hence $p_i(f)$ is in fact in $\FF_q$. Moreover, we have $p_i(fg) = p_i(f)+p_i(g)$ for any monic $f,g\in \FF_q[T]$. If $f(T) = \det(I_n T-g)$ is the characteristic polynomial of a matrix $g$, then $p_i(f) = \Tr (g^i)$. 

Let $F(T):=\sum_{i=0}^{k} a_i T^i$ be the product of all monic irreducible polynomials in $\FF_q[T]$ of degree at most $n$ (so $a_k=1$, $a_0=0$). Then, for any $\lambda \in \cup_{j=1}^n \FF_{q^j}$,
\[
F(\lambda) = \sum_{i=1}^{k} a_i \lambda^i= 0.
\]
Hence $\sum_{i=1}^{k} a_i p_i(f) = 0$ for any $f$ with $\deg f \le n$ and

\begin{equation}
    \sum_{i=1}^{k} a_i \Tr(g^i) = \sum_{i=1}^{k} a_i p_i(\det(I_n T-g))=0
\end{equation}
for all $g \in \glnq$. Finally, $k=\deg F  \asymp q^n$ by the Prime Polynomial Theorem (see Lemma \ref{lem:pnq}). If $k$ happens to be divisible by $\charfq$, we can replace $F$ by $TF$. 

The question whether the range of Theorem \ref{thm:matrix} can be extended to $\log k \le c_q n$ for some constant $c_q>0$, for which we believe the answer is negative, remains open and is not explored here.

For certain values of $k$, for example, primes, we can go beyond the range $\log k =o(n^2)$ of Theorem \ref{thm:matrixsingle} using the following arithmetic criterion.
\begin{thm}\label{thm:crit}
Let $g \in \glnq$ be chosen uniformly at random. Let $k = k(n)$ be a positive integer. Suppose 
\begin{equation}\label{eq:sumdiv} \sum_{\substack{12 \log_q n < d\le n \\ \gcd(k,q^d-1) < q^{d/3}}}d^{-1}\to \infty
\end{equation}
as $n \to \infty$. Then, the distribution of $\Tr(g^k)$ tends to the uniform distribution on $\FF_q$ as $n$ tends to $\infty$.
\end{thm}
As we shall see later, Theorem \ref{thm:matrixsingle} is a consequence of Theorem \ref{thm:crit}.

 \subsection{Comparison with random matrix theory}
 Our investigation was motivated by results in random matrix theory, although we do not use any techniques from this area. Let $\text{U}_n(\CC)$ be the group of $n\times n$ unitary matrices over complex numbers, endowed with Haar measure of total mass 1. A classical result of Diaconis and Shahshahani \cite{Diaconis1994} states that for any $k \ge 1$, the vector $X_k=(\Tr(U),\Tr(U^2)/\sqrt{2},\ldots, \Tr(U^k)/\sqrt{k})$ converges in distribution to $Y_k=(Z_1, Z_2,\ldots, Z_k)$, where $(Z_j)_{j=1}^k$ are independent standard complex Gaussians. Johansson \cite[Thm.~2.6]{Johansson1997} showed that the rate of convergence of a linear combination of $\Tr(U^i)$ to a suitable Gaussian is superexponential in total variation distance. In \cite{johansson2021multivariate}, Johansson and Lambert extended \cite{Johansson1997} to the total variation distance of $X_k$ from $Y_k$ uniformly for $k \ll n^{2/3-\varepsilon}$, and in \cite{courteaut2021multivariate}, Courteaut and Johansson established similar results for orthogonal and symplectic groups. In a recent work, Courteaut, Johansson, and Lambert \cite{courteaut2022berry} studied the convergence of $\Tr(U^k)/\sqrt{k}$ to $Z_k$ as $k$ varies, obtaining, among other results, that the distance goes to $0$ for any $k$ in the range $1 \le k < n$. As for the complementary range, Rains \cite{Rains1997} proved that there is a stabilizing phenomenon once $k \ge n$: The eigenvalues of $U^k$ become distributed as $n$ independent uniform random variables on the unit circle, and in particular $\Tr(U^k)/\sqrt{n}$ tends in distribution to the standard complex Gaussian.

 \subsection{Symmetric function perspective}
Let us also formulate a natural problem in symmetric functions that will turn out to share strong similarities with the equidistribution of $\Tr(g^k)$. Let $e_k(t_1,\ldots,t_n)$ be the $k$th elementary symmetric polynomial
\begin{equation}
    e_k(t_1,\ldots,t_n) := \sum_{1 \le j_1 < j_2<\cdots <j_k \le n} t_{j_1}t_{j_2}\cdots t_{j_k}
\end{equation}
and $p_k(t_1,\ldots,t_n)$ be the $k$th power sum symmetric polynomial
\[p_k(t_1,\ldots,t_k):=\sum_{i=1}^{n} t_i^k.\]
Let $\mathbf{X} = (X_i)_{i=1}^{n}$ be $n$ $\barfq$-valued random variables such that $e_1(\mathbf{X}),\ldots,e_n(\mathbf{X})$ are independent and uniformly distributed on $\FF_q$. By Newton's identities, $p_k(\mathbf{X})$ must also be $\FF_q$-valued. We ask, is $p_k(\mathbf{X})$ close to uniform?

Here is one way to construct such a sequence $\mathbf{X}$. Taking $a_1,\ldots,a_n$ to be independent uniform random variables on $\FF_q$, the polynomial $T^n - a_1T^{n-1}+a_2T^{n-2} \mp \ldots + (-1)^na_n$ is distributed uniformly among $\Mnq \subseteq \mathbb{F}_q[T]$, the subset of monic polynomials of degree $n$. Setting $(X_i)_{i=1}^{n}$ to be its $n$ roots in some order, we have that $e_j(\mathbf{X}) = a_j$ satisfy the conditions above, and $p_i(f)$ defined in \eqref{eq:pk} coincides with $p_i(\mathbf{X})$.
We prove the following.
\begin{thm}\label{thm:symmetric}
Let $n \ge 1$. Let $k = k(n)$ be a positive integer. Suppose $\mathbf{X}=(X_i)_{i=1}^{n}$ are $n$ random variables such that $(e_1(\mathbf{X}),\ldots,e_n(\mathbf{X}))$ has the uniform distribution on $(\FF_q)^n$. Then:
\begin{enumerate}
    \item If $\log k =o(n^2)$ then the distribution of $p_k(\mathbf{X})$ tends to the uniform distribution on $\FF_q$.
    \item Let $a_1,\ldots,a_k \in \FF_q$. If $\log k = o(n)$, $a_k \neq 0$ and $\charfq \nmid k$ then the distribution of $\sum_{i=1}^{k} a_i p_i(\mathbf{X})$ tends to the uniform distribution on $\FF_q$.
    \item If the sum in \eqref{eq:sumdiv} diverges then the distribution of $p_k(\mathbf{X})$ tends to the uniform distribution on $\FF_q$.
\end{enumerate}
\end{thm}
Theorem~\ref{thm:symmetric} is about $p_k(f)$ when we choose $f$ uniformly at random from $\Mnq$, while Theorems~\ref{thm:matrixsingle}--\ref{thm:crit} are about $p_k(f)$ for a polynomial $f$ drawn from the space of possible characteristic polynomials of a matrix from $\glnq$ endowed with the uniform measure. As proved in \cite[Thm.~1.4]{Gorodetsky2021}, the total variation distance of the law of the first $k$ next-to-leading coefficients of $\det(I_n T-g)$  from the uniform distribution on $\FF_q^{k}$ tends to $0$ as $n \to \infty$ for $k$ as large as $n-o(\log n)$, so the setup of Theorem~\ref{thm:symmetric} is not that different from the setup of Theorems~\ref{thm:matrixsingle}--\ref{thm:crit}.

\subsection{Cancellation in character sums of high conductor}\label{sec:char}
Given $n \ge 0$ we denote by $\Mnq\subseteq \FF_q[T]$ the subset of monic polynomials of degree $n$. We denote by $\Mq=\cup_{n \ge 0} \Mnq$ the set of monic polynomials in $\FF_q[T]$. We denote by $\Pnq \subseteq \Mnq$ the set of irreducible polynomials of degree $n$ and let $\Pq = \cup_{n \ge 0} \Pnq$. Throughout, the letter $P$ is reserved for elements of $\Pq$.

Given $k \ge 1$ and a nontrivial additive character $\psi \colon \FF_q \to \CC$, we define a function $\chi_{k,\psi}\colon \Mq \to \CC$ by 
\[\chi_{k,\psi}(f) := \begin{cases} \psi(p_{-k}(f)), & \text{if $(f,T)=1$},\\
0, & \text{otherwise.}
\end{cases}
\]
\begin{remark}\label{rem:negative}
	If $i$ is a negative integer, then $p_i(f)$, as defined in  \eqref{eq:pk}, is well defined as long as $T \nmid f$. Moreover, the usual properties are preserved: $p_i(fg)=p_i(f)+p_i(g)$ if $T\nmid fg$; if $f(T)=\deg(I_n T-g)$ then $p_i(f)=\Tr(g^i)$. When $T \nmid f$, $p_i(f)=p_{-i}( f(1/T)T^{\deg f}/f(0))$, where $f(1/T)T^{\deg f}/f(0)$ is a monic polynomial, proving $p_i(f) \in \FF_q$ as well.
\end{remark}
We show in Lemma \ref{lem:invchar} that the function $\chi_{k,\psi}$ is a primitive Dirichlet character modulo $T^{k+1}$. The following theorem is the main component behind the proof of  Theorem \ref{thm:matrixsingle}.
\begin{thm}\label{thm:canc}
We have, uniformly for $n\ge 1$ and $k \ge 1$,
 \[ q^{-n}\sum_{f \in \Mnq} \chi_{k,\psi}(f) \ll\frac{1+ \log_q n +\sqrt{\log_q k}}{n}.\]
 In particular, we have cancellation when $\log k = o(n^2)$ as $n \to \infty$.
\end{thm}
Note that the range $\log k = o(n^2)$ is optimal because for $k=\prod_{i=1}^{n}(q^i-1)$ we have $\gamma^{-k}=1$ for every $\gamma \in \cup_{i=1}^{n} \FF_{q^i}^{\times}$ and thus $p_{-k}(f) \equiv n$ on $\{f \in \Mnq : T \nmid f\}$. Hence, there is no cancellation for such $k$. We do not know if cancellation for $\log k=o(n^2)$ persists if we restrict to $\Pnq$ instead of $\Mnq$; the proof of Theorem \ref{thm:canc} relies on summing over \textit{all} degree-$n$ polynomials. 

It is instructive to compare Theorem \ref{thm:canc} to results about character sums for integers and polynomials.
We switch temporarily to the integer setting. Let $\chi$ be a nonprincipal Dirichlet character modulo $m$ and consider the sum $S_{x,\chi}:=\sum_{n \le x} \chi(n)$ as $x \to \infty$. Montgomery and Vaughan \cite[Lem.~2]{Montgomery1977} proved under the generalized Riemann hypothesis for $L(s,\chi)$ that if $m \ge x$ and $y \in [(\log m)^4,x]$ is a parameter, then 
\begin{equation}\label{eq:MV} S_{x,\chi} =\sum_{\substack{n\le x \\ p \mid n \implies p \le y}} \chi(n)+ O(xy^{-1/2}(\log m)^4).
\end{equation}
Recall $\sum_{n \le x: \, p \mid n \implies p \le y}1$ is $o(x)$ if $\log x/\log y \to \infty$ \cite[eq.~(1.9)]{deb}.
Taking $y = (\log m)^9$, we see that $S_{x,\chi}$ exhibits cancellation if $\log \log m=o(\log x)$. Granville and Soundararajan improved  the error term in \eqref{eq:MV} and also showed that this range is optimal \cite[Cor.~A]{GS}, in the sense that for any given $A>0$ and for any prime $m$, there exists a nonprincipal character $\chi \bmod m$ with $|S_{x,\chi}| \gg_A x$, where $x = \log^ A m$.

Much less is known unconditionally. Burgess proved $S_{x,\chi}$ exhibits cancellation when $m\le x^{3-\varepsilon}$, and this can be  extended to $m \le x^{4-\varepsilon}$ if $m$ is cubefree \cite{Burgess1,Burgess2}. When the conductor is smooth, better results exist \cite[Ch.~12]{IK}. In particular, if $m=p^r$, then Banks and Shparlinski \cite{Banks} showed one has cancellation in the range  $\log m = o( (\log x)^{3/2})$ (if $r \ge C$ and $p\le x^c$ for some absolute $C>0$ and $c>0$); this improved earlier work of Postnikov  \cite[Thm.~12.16]{IK}. It is worthwhile to recall  $\sum_{n \le x}n^{it}$ exhibits cancellation when $\log (|t|+2) =o( (\log x)^{3/2})$, a result due to Vinogradov \cite[Cor.~8.26]{IK}. The bound we prove for the sum of $f\mapsto \chi_{k,\psi}(f)$ also holds for the function $f\mapsto \psi(p_k(f))$, which is an analogue of $n\mapsto n^{it}$, see Remark \ref{rem:kminusk} and \cite[Lem.~2.2]{Gorodetsky2021}.

Now let us return to the polynomial setting. The generalized Riemann hypothesis in $\FF_q[T]$ is a seminal theorem due to A.~Weil \cite{Weil}. Let $\chi$ be a nonprincipal  Dirichlet character modulo a polynomial $Q \in \FF_q[T]$.
In \cite[Thm.~3]{Bhowmick2015}, Bhowmick and L\^{e} adapted \eqref{eq:MV} to the polynomial setting, proving unconditionally that
\begin{equation}\label{eq:bho3} \sum_{f \in \Mnq} \chi(f) = \sum_{\substack{f \in \Mnq\\ P \mid f \implies \deg P \le m}} \chi(f) + O\left( q^{n-\frac{m}{2}}\deg Q \right)
\end{equation}
holds for $m \in [2\log_q \deg Q ,n]$. Taking $m=\lceil 4\log_q \deg Q \rceil$ shows that, as $n \to \infty$, $\sum_{f \in \Mnq} \chi(f)$ exhibits cancellation when $\log_q \deg Q = o(n)$. (For a self-contained bound on the sum in the right-hand side of \eqref{eq:bho3}, see Lemma \ref{lem:sieve}.)

The range $\log k=o(n^2)$ in Theorem \ref{thm:canc} is far beyond the ranges that the generalized Riemann hypothesis implies, and heavily exploits a special symmetry satisfied by $\chi_{k,\psi}$, which is shown in Lemma \ref{lem:sym}. We are not aware of any other \textit{explicit} family of characters, in either integers or polynomials, where the generalized Riemann hypothesis implies cancellation when the conductor exceeds the Montgomery--Vaughan range $\log \log m = o(\log x)$ in $\ZZ$ or $\log \deg Q = o(n)$ in $\FF_q[T]$.

In \S\ref{sec:app} we prove a new bound on general character sums in function fields which is not used in the paper.
\subsection{Structure of the paper}

In Section \ref{sec:reductions}, we discuss the connection between the distribution of traces and character sums in more detail, and bound the total variation distance between our respective distributions and the uniform distribution by corresponding character sums. By doing so, we reduce Theorem \ref{thm:matrixsingle} and the first part of Theorem \ref{thm:symmetric} to Theorem \ref{thm:canc}, and Theorem \ref{thm:matrix} and the second part of Theorem \ref{thm:symmetric} to obtaining bounds for the character sum in \eqref{eq:bho3}. The latter is quite straightforward as shown in Lemma \ref{lem:char}, while the former requires more careful treatment. We prove Theorem \ref{thm:canc} in Section \ref{sec:canc}; we consider this to be the technical part of the paper. 
We start with observing that there is underlying symmetry in sums of $\chi_{k,\psi}$ against primes. To see this, in Lemma \ref{lem:sym}, we prove
\[ \sum_{f \in \Mnq} \Lambda(f)\chi_{k,\psi}(f) =\sum_{f \in \Mnq} \Lambda(f)\chi_{k^\prime,\psi}(f),\]
where $\Lambda(f)$ is the function field von Mangoldt function, and $k^\prime = \gcd(k,q^n-1)$. This identity allows us to improve the Weil bound
\[ \bigg|\sum_{f \in \Mnq} \Lambda(f)\chi_{k,\psi}(f)\bigg| \le q^{\frac{n}{2}} k \]
to
\begin{equation}\label{eq:impweil}
\bigg|\sum_{f \in \Mnq} \Lambda(f)\chi_{k,\psi}(f)\bigg| \le q^{\frac{n}{2}}k^\prime =  q^{\frac{n}{2}}\gcd(k,q^n-1)
\end{equation}
as shown in Corollary \ref{cor:sym}. This is helpful when $\gcd(k,q^n-1)$ is small. Motivated by Montgomery and Vaughan \cite[Lem.~2]{Montgomery1977} given $S \subseteq \{1,\ldots,n\}$, we write
\begin{equation}\label{eq:stra}
q^{-n} \sum_{f \in \Mnq}\chi_{k,\psi}(f)=q^{-n} \sum_{\substack{f \in \Mnq\\ P \mid f \implies \deg P \not\in S}} \chi_{k,\psi}(f)+q^{-n} \sum_{\substack{f\in \Mnq\\\exists P \mid f\text{ such that }\deg P \in S}}\chi_{k,\psi}(f).
\end{equation}
In our case, choosing
\[
S=\{\lceil 12\log_qn \rceil \le d \le n: \gcd(k,q^d-1)<q^{d/3}\}
\]
allows us to bound the second sum on the right-hand side of \eqref{eq:stra} using \eqref{eq:impweil} (see Lemma \ref{lem:MV}) while a sieve bound (Lemma \ref{lem:sieve}) bounds the first sum on the right-hand side of \eqref{eq:stra}. This strategy ends up proving the following criterion.
\begin{prop}\label{prop:crit}
Let $n \ge 1$ and $k \ge 1$. Let $\psi\colon \FF_q \to \CC^{\times}$ be a nontrivial additive character. We have
\begin{equation}\label{eq:crit} q^{-n}\sum_{f \in \Mnq} \chi_{k,\psi}(f) \ll n^{-1} + \exp\bigg( - \sum_{\substack{12\log_q n < d \le n \\ \gcd(k,q^d-1)<q^{d/3}}} d^{-1}\bigg).
\end{equation}
In particular, a sufficient criterion for cancellation is that the sum in \eqref{eq:sumdiv} diverges as $n \to \infty$.
\end{prop}
In Section \ref{sec:reductions}, we reduce Theorem \ref{thm:crit} and the third part of Theorem \ref{thm:symmetric} to Proposition \ref{prop:crit}. To deduce Theorem \ref{thm:canc} from Proposition \ref{prop:crit}, we prove a sharp upper bound on $\sum_{L<d \le 2L} \log \gcd(k,q^d-1)$ in Lemma \ref{lem:Blk}, that may be of independent interest.

The criterion above is particularly interesting because it allows us to exhibit cancellation on the left-hand side of \eqref{eq:crit} for arbitrarily large $k$, for example, whenever $k$ is a prime, or a product of a bounded number of primes, the sum in \eqref{eq:sumdiv} diverges simply because the following shorter sum does:
\[ \sum_{\substack{12 \log_q n < d\le n \\ \gcd(k,q^d-1)=1}} d^{-1}.\]
\begin{remark}
Proposition \ref{prop:crit}, and hence Theorem \ref{thm:canc}, apply as is to $\sum_{f \in \Mnq} \mu(f)\chi_{k,\psi}(f)$ where $\mu$ is the M\"obius function, see Remark \ref{rem:mo}. Recall $\mu$ is multiplicative with $\mu(P)=-1$ and $\mu(P^e)=0$ for $e \ge 2$. However, we do not know that $\log k =o(n^2)$ is optimal in this case.
\end{remark}
\begin{remark}
For every $\varepsilon>0$, there are examples where $\log k \sim \varepsilon n^2$ for which Proposition \ref{prop:crit} does not yield cancellation, for example, $k=\prod_{i=1}^{\lfloor \varepsilon' n\rfloor}(q^i-1)$ where taking logarithm shows $\varepsilon'>0$ is defined via $\varepsilon=\varepsilon'^2 (\log q)/2$.
\end{remark}
\begin{remark}
It is natural to ask whether analogues of Theorems \ref{thm:matrixsingle} and \ref{thm:matrix} hold for other groups. Consider for instance $\mathrm{U}(n,q)$, the unitary group over $\mathbb{F}_{q^2}$. Then, as in Section \ref{sec:reductions}, we are led to consider the sums $\sum_{f\in \mathcal{M}_{n,q^2}} \mathrm{P}_{\mathrm{U}}(f) \chi(f)$ where $\mathrm{P}_{\mathrm{U}}(f):= \PP_{g \in \mathrm{U}(n,q)}( \det(I_n T-g) = f)$ is a function studied in \cite[\S5]{Gorodetsky2021} and $\chi$ is a Dirichlet character depending on $k$ (if we are in the situation of Theorem~\ref{thm:matrixsingle}) or on $a_1,\ldots,a_k$ (if we are in the situation of Theorem~\ref{thm:matrix}). The support of $\mathrm{P}_{\mathrm{U}}$ is restricted since the characteristic polynomial of any matrix from $\mathrm{U}(n,q)$ is \textit{unitary self-reciprocal}, that is, it belongs to
\[\mathcal{M}_{n,q^2}^{usr}:=\left\{ f \in \mathcal{M}_{n,q^2}: f(0)\neq 0,\, f(T) = T^{\deg f} \frac{\sigma(f(1/T))}{\sigma(f(0))}\right\}, \qquad \sigma\big(\sum c_i T^i\big) := \sum c_i^q T^i.\]
Using the strategy in Section \ref{sec:reductions} together with \cite[Thm.~5.10]{Gorodetsky2021}, the  problem ultimately reduces to bounding character sums over $\mathcal{M}_{n,q^2}^{usr}$, $\sum_{f \in \mathcal{M}_{n,q^2}^{usr}} \chi(f)$. 	We expect that it is possible to adapt our proofs to bound such sums with additional work, although various complications arise, and this problem deserves further study.  For a review of character sums over such sets, we refer the reader to \cite[\S5]{Gorodetsky2021}.
\end{remark}

\section{Reductions}\label{sec:reductions}
\subsection{An involution}
Given $f \in \Mq$ we use the notation $p_i(f)$ introduced in \eqref{eq:pk}.
Given $a_1,\ldots,a_k \in \FF_q$ and an additive character $\psi \colon \FF_q \to \CC^{\times}$, define a function $\xi_{\mathbf{a},\psi}\colon \Mq \to \CC$ by 
\[ 
\xi_{\mathbf{a},\psi}(f) := \psi\bigg(\sum_{i=1}^{k} p_i(f) a_i\bigg).
\]
The function $\xi_{\mathbf{a},\psi}$, sometimes called a short interval character, is closely related to a certain Dirichlet character modulo $T^{k+1}$ as observed by Hayes \cite[pp.~115--116]{hayes1965distribution} and Keating and Rudnick \cite[\S4.1]{KR1}\cite[p.~389]{KR2}. Let us explain this idea. We let $\Mqt:=\{ f \in \Mq : (f,T)=1\}$ (this notation, borrowed from \cite[\S3.1]{Gorodetsky2021}, is motivated by the fact that the characteristic polynomial of any matrix from $\cup_{n\ge 0}\glnq$ lies in $\Mqt$). We define an involution $\iota$ on $\Mqt$ by
\[\iota(f) := f(1/T)T^{\deg f}/f(0).\]
If we factorize $f\in \Mqt$ as $f(T)=\prod_{j=1}^{\deg f}(T-\lambda_j)$ over $\barfq$, we find that
\[ \iota(f) = \prod_{j=1}^{\deg f}(T-\lambda_j^{-1})\]
so $p_k(f) = p_{-k}(\iota(f))$ (see Remark~\ref{rem:negative} for the definition of $p_{-k}$).
For any function $\alpha \colon \Mq\to \CC$, we define $\iota_\alpha \colon \Mq \to \CC$ by
\[\iota_\alpha(f) := \begin{cases} \alpha(\iota(f)) & \text{if } f \in \Mqt,\\
0 & \text{if }f \not\in\Mqt. 
\end{cases}
\]
We say that a function $\alpha \colon \Mq \to \CC$ is completely multiplicative if $\alpha(fg)=\alpha(f)\alpha(g)$ holds for all $f,g \in \Mq$ and $\alpha(1)=1$. If $\alpha$ is completely multiplicative, then so is $\iota_\alpha$ because $\iota$ is (when extended to $\Mq$ via $\iota(T)=0$).
Given an arithmetic function $\alpha \colon \Mq \to \CC$, we use the notation
\[ S(n,\alpha) := \sum_{f \in \Mnq} \alpha(f).\]
Letting 
\[ \chi_{0}(f) := \mathbf{1}_{T \nmid f},\]
we have
\begin{align}\label{eq:sscon}
S(n,\alpha \cdot \chi_{0}) = S(n,\iota_\alpha)
\end{align}
and, since $S(n,\alpha \cdot \chi_{0}) = S(n,\alpha)-\alpha(T)S(n-1,\alpha)\mathbf{1}_{n \ge 1}$ holds for completely multiplicative $\alpha$,
\begin{align}\label{eq:sscon2}
 S(n,\alpha) = \sum_{i=0}^{n} S(i,\alpha \cdot \chi_{0}) \alpha(T)^{n-i}= \sum_{i=0}^{n} S(i,\iota_\alpha)\alpha(T)^{n-i}.
\end{align}
A variant of the following lemma was established in \cite[Lem.~2.2]{Gorodetsky2021}.
\begin{lem}\label{lem:invchar}
Let $a_1,\ldots,a_k \in \FF_q$ with $a_k \neq 0$ and $\charfq \nmid k$. Let $\psi\colon \FF_q \to \CC^{\times}$ be a nontrivial additive character, and let $\alpha = \xi_{\mathbf{a},\psi}$. Then, $\iota_\alpha$ is a primitive Dirichlet character modulo $T^{k+1}$.
\end{lem}
\begin{proof}
We know $\iota_\alpha$ is completely multiplicative, vanishes at multiples of $T$ and $\iota_\alpha(1) = 1$. It remains to show that it only depends on the residue of the input modulo $T^{k+1}$ and that it is primitive.
Newton's identities yield that $p_i(f)$ is a function of the $i$ first next-to-leading coefficients of $f$, where the $j$th ($j \ge 1$) next-to-leading coefficient of $T^n+\sum_{i=0}^{n-1} a_iT^i$ is defined to be $a_{n-j}$ if $j \le n$ and to be $0$ otherwise. Hence, $\alpha = \xi_{\mathbf{a},\psi}(f)$ only depends on the $k$ first next-to-leading coefficients of $f$. Since by definition $\iota(f)$ reverses the coefficients of $f$ and normalizes by $f(0)$, it follows that $\iota_\alpha(f)$ depends only on the last $k+1$ coefficients of $f$, that is, on $f \bmod T^{k+1}$.
Finally, $\iota_\alpha(T^{k+1}-c)=1$ for every $c \in \FF_q^{\times}$ while a short computation using Newton's identities shows $\iota_\alpha(T^k-c)=\psi(ka_k/c)$ is not equal to $1$ for suitable $c$.
\end{proof}
For $k \ge 1$ and an additive character $\psi \colon \FF_q \to \CC$, with a slight abuse of notation, let 
\begin{align}
\xi_{k,\psi}(f) &:=  \psi(p_k(f)),
\end{align}
which coincides with $\xi_{\mathbf{a},\psi}$ for $a_1=\ldots=a_{k-1}=0$, $a_k=1$.
By definition, we have
\[
\chi_{k,\psi}= \iota_{\xi_{k,\psi}} = \iota_{\xi_{k,\psi}\cdot \chi_{0}}. 
\]
\begin{remark}\label{rem:kminusk}
We have $S(n,\xi_{k,\psi})=\sum_{i=0}^{n} S(i,\chi_{k,\psi})$ by \eqref{eq:sscon2}. In particular, $q^{-n}S(n,\chi_{k,\psi}) \ll (1+\log_q n + \sqrt{\log_q k})/n$ implies $q^{-n}S(n,\xi_{k,\psi}) \ll (1+\log_q n + \sqrt{\log_q k})/n$ and vice versa.
\end{remark}
\subsection{Reduction of Theorem \ref{thm:symmetric} to character sum estimates}
Recall that $\mathbf{X} = (X_i)_{i=1}^{n} \in \barfq^n$ are $n$ random variables such that $e_1(\mathbf{X}),\ldots,e_n(\mathbf{X})$ are independent and uniformly distributed on $\FF_q$, and that one can take $X_i$ to be the zeros of a polynomial $f$ chosen uniformly at random from $\Mnq$. 

Let $a_1,\ldots,a_k \in \FF_q$ and denote by $\widehat{\FF_q}$ the group of $q$ additive characters from $\FF_q$ to $\CC^{\times}$, with $\psi_0$ denoting the trivial character.
By Fourier analysis on $\FF_q$, given $x \in \FF_q$,
\begin{multline}
\PP(\sum_{i=1}^{k}a_i p_i(\mathbf{X})=x) = \PP_{f \in \Mnq}(\sum_{i=1}^{k}a_i p_i(f)=x)\\ = q^{-1} \sum_{\psi\in \widehat{\FF_q}} \overline{\psi}(x)\EE_{f \in \Mnq} \psi(\sum_{i=1}^{k} a_i p_i(f))=q^{-1} \sum_{\psi\in \widehat{\FF_q}} q^{-n} \overline{\psi}(x) S(n,\xi_{\mathbf{a},\psi}).
\end{multline}
Since $S(n,\xi_{\mathbf{a},\psi_0})=q^n$, the triangle inequality implies
 \begin{equation}\label{eq:triangle ineq}
 	\sum_{x \in \FF_q} \big| \PP(\sum_{i=1}^{k} a_i p_i(\mathbf{X})=x) - q^{-1} \big|  \le \sum_{\psi_0 \neq \psi \in \widehat{\FF_q}} |q^{-n}S(n,\xi_{\mathbf{a},\psi})|.
 \end{equation}
 By \eqref{eq:triangle ineq}, \eqref{eq:sscon2}, and the triangle inequality,
\begin{equation}\label{eq:triangle ineq2} \sum_{x \in \FF_q} \big| \PP(\sum_{i=1}^{k} a_i p_i(\mathbf{X})=x) - q^{-1} \big|  \le\sum_{j=0}^{n}q^{j-n} \sum_{\psi_0 \neq \psi \in \widehat{\FF_q}} |q^{-j}S(j,\iota_{\xi_{\mathbf{a},\psi}})|.
\end{equation}
A special case of \eqref{eq:triangle ineq2} is 
\begin{equation}\label{eq:triangleineqspecial}
\sum_{x \in \FF_q} \big| \PP(p_k(\mathbf{X})=x) - q^{-1} \big|  \le\sum_{j=0}^{n}q^{j-n} \sum_{\psi_0 \neq \psi \in \widehat{\FF_q}} |q^{-j}S(j,\chi_{k,\psi})|.
\end{equation}
The first part of Theorem \ref{thm:symmetric} is immediate from \eqref{eq:triangleineqspecial} and Theorem \ref{thm:canc}. The third part follows from \eqref{eq:triangleineqspecial} and Proposition \ref{prop:crit}. For the second part of Theorem \ref{thm:symmetric}, we use \eqref{eq:triangle ineq2} and observe $\iota_{\xi_{\mathbf{a},\psi}}$ is a nonprincipal Dirichlet character modulo $T^{k+1}$ by Lemma \ref{lem:invchar}, and the result follows from applying the following lemma.

 \begin{lem}[Bhowmick--L\^{e}]\label{lem:char}
Let $\chi\colon \Mq \to \CC$ be a nonprincipal Dirichlet character modulo $Q$. We have
\[ q^{-n}S(n,\chi) \ll \frac{1+\log_q (1+\deg Q)}{n+1}.\]
\end{lem}
Lemma \ref{lem:char} follows from \eqref{eq:bho3} with $m=\lceil 2\log_q(n(1+\deg Q))\rceil$ (see Section \ref{sec:prooflemchar} for details).
\subsection{Reductions of Theorems \ref{thm:matrixsingle}--\ref{thm:crit} to character sum estimates}
We define an arithmetic function $\gfnc \colon \Mq \to \CC$ as follows. If $f \in \Mq$ and $n = \deg f$, then
\begin{equation}
\gfnc(f) := \PP_{g \in\glnq}(\det(I_{n} T - g) = f),
\end{equation}
where $\glnq$ is endowed with the uniform measure. It follows from the work of Reiner \cite{Reiner1961} and Gerstenhaber \cite{Gerstenhaber1961} that $\gfnc$ is multiplicative (in the sense that $\gfnc(fg)=\gfnc(f)\gfnc(g)$ when $\gcd(f,g)=1$) and supported on $\Mqt$. Their works also show that on prime powers it is given by $\gfnc(T^e)=0$ and
\begin{equation}\label{eq:gfncpp} \gfnc(P^e) = q^{-e\deg P} \prod_{i=1}^{e} (1-q^{-i\deg P})^{-1},
\end{equation}
where $P \in \Pq \setminus \{T\}$ and $e \ge 1$
(cf.~\cite[Thm.~3.3]{Gorodetsky2021}). The following identity is a quick consequence of \eqref{eq:gfncpp} (it can also be derived directly from \cite[Thm.~3.4]{Gorodetsky2021}). Recall $|f|=q^{\deg f}$.
\begin{lem}\label{lem:iden}
We have $\gfnc = \alpha_1 * \alpha_2$ where $\alpha_1(f) = |f|^{-1}\cdot \mathbf{1}_{T \nmid f}$ and $\alpha_2(f) = \gfnc(f)/|f|$.
\end{lem}
Let  $a_1,\ldots,a_k \in \FF_q$. By Fourier analysis on $\FF_q$, given $x \in \FF_q$, we have
\[\PP_{g \in \glnq}\big(\sum_{i=1}^{k} a_i\Tr(g^i)=x\big)= q^{-1} \sum_{\psi\in  \widehat{\FF_q}} \EE_{g \in \glnq}\overline{\psi}(x)\psi( \sum_{i=1}^{k}a_i\Tr(g^i)).\]
The trivial character $\psi_0$ contributes $1/q$ to the right-hand side. Recall $p_k(f)$ ($f \in \Mq$) is defined in \eqref{eq:pk}. Since $p_i(\deg(I_nT-g)) = \Tr(g^i)$,
\[\PP_{g \in \glnq}\big(\sum_{i=1}^{k} a_i\Tr(g^i)=x\big) -q^{-1}= q^{-1} \sum_{\psi_0 \neq \psi \in \widehat{\FF_q}}\overline{\psi}(x) \sum_{f \in \Mnq} \gfnc(f) \xi_{\mathbf{a},\psi}(f).\]
Summing over $x$ and using the triangle inequality gives 
\begin{equation}\label{eq:foucon}
\sum_{x \in \FF_q}\big|\PP_{g \in \glnq}\big(\sum_{i=1}^{k} a_i\Tr(g^i)=x\big) - q^{-1}\big| \le \sum_{\psi_0 \neq \psi \in \widehat{\FF_q}} |S(n, \gfnc \cdot \xi_{\mathbf{a},\psi})|.
\end{equation}
In \cite[Thm.~3.6]{Gorodetsky2021}, the first author and Rodgers showed that
\begin{equation}\label{eq:gororodg}
|S(n, \gfnc \cdot \xi_{\mathbf{a},\psi})| \le q^{-\frac{n^2}{2k}+O_q(n)}.
\end{equation}
For fixed $k$ this is essentially optimal up to constants, because $|S(n, \gfnc \cdot \xi_{\mathbf{a},\psi})|$ cannot decay faster than exponentially in $n^2$ due to $|\glnq| = q^{\Theta(n^2)}$. 
In this paper, we focus on the range of cancellation rather than the rate of cancellation, and in this aspect, we can do better. Observe that from Lemma \ref{lem:iden},
\[ S(n,\gfnc \cdot \xi_{\mathbf{a},\psi}) = q^{-n} \sum_{i+j=n} \sum_{f \in \mathcal{M}_{i,q}, \, T \nmid f} \xi_{\mathbf{a},\psi}(f) \sum_{g \in \mathcal{M}_{j,q}} \gfnc(g)\xi_{\mathbf{a},\psi}(g).\]
Since $\gfnc$ is a probability measure on $\mathcal{M}_{j,q}$, the triangle inequality yields 
\begin{equation}\label{eq:triang} |S(n,\gfnc \cdot \xi_{\mathbf{a},\psi})| \le  q^{-n} \sum_{i=0}^{n} |S(i,\xi_{\mathbf{a},\psi} \cdot \chi_{0}) |.
\end{equation}
From \eqref{eq:triang}, \eqref{eq:foucon}, and \eqref{eq:sscon},
\begin{equation}\label{eq:done}
\sum_{x \in \FF_q}\big|\PP_{g \in \glnq}\big(\sum_{i=1}^{k} a_i\Tr(g^i)=x\big) - q^{-1}\big| \le q^{-n} \sum_{\psi_0 \neq \psi \in \widehat{\FF_q}} \sum_{i=0}^{n}|S(i, \iota_{\xi_{\mathbf{a},\psi}})|.
\end{equation}
Similarly to the symmetric function case, we see that Theorem \ref{thm:matrix} follows from \eqref{eq:done} and Lemma \ref{lem:char}. 

In the special case $a_1=\ldots=a_{k-1}$ and $a_k=1$, we have
\begin{equation}\label{eq:done2}
\sum_{x \in \FF_q}\big|\PP_{g \in \glnq}\big(\Tr(g^k)=x\big) - q^{-1}\big| \le q^{-n} \sum_{\psi_0 \neq \psi \in \widehat{\FF_q}} \sum_{i=0}^{n}|S(i, \chi_{k,\psi})|.
\end{equation}
Theorem \ref{thm:matrixsingle} follows from \eqref{eq:done2} and Theorem \ref{thm:canc}. Theorem \ref{thm:crit} follows from \eqref{eq:done2} and Proposition \ref{prop:crit}. 

\section{Proof of Theorem \ref{thm:canc}}\label{sec:canc}
The von Mangoldt function $\Lambda \colon \Mq \to \CC$ is defined as
\begin{equation}\label{def:lambda}
\Lambda(f) = \begin{cases} \deg P & \text{if }f = P^k, P \in \Pq,\, k \ge 1, \\ 0 & \text{otherwise.}\end{cases}
\end{equation}
Gauss' identity \cite[Thm.~2.2]{Rosen2002} states that
\begin{equation}\label{eq:gauss}
\sum_{f \in \Mnq} \Lambda(f) = \sum_{d \mid n} |\mathcal{P}_{d,q}| d =q^n,
\end{equation}
and is known to imply the following.
\begin{lem}\cite[Lem.~4]{Pollack}\label{lem:pnq}
We have $q^n/n - 2q^{n/2}/n \le |\Pnq| \le q^n/n$ for $n \ge 1$. 
\end{lem}
Given a Dirichlet character $\chi$ its $L$-function is defined as
\begin{equation}\label{eq:euler}
    L(u,\chi) = \sum_{f \in \Mq} \chi(f) u^{\deg f} = \prod_{P \in \Pq}(1-\chi(P)u^{\deg P})^{-1},
\end{equation}
which converges absolutely for $|u|<1/q$. If $\chi$ is a nonprincipal character modulo $Q$, then $L(u,\chi)$ is a polynomial of degree at most $\deg Q - 1$ as follows from the orthogonality relations for $\chi$. We set
\[ d(\chi):=\deg L(u,\chi) < \deg Q.\]
\begin{thm}[Weil's RH \cite{Weil}]\label{thm:weil} Let $\chi$ be a nonprincipal Dirichlet character. Factoring $L(u,\chi)$ as
\begin{equation}\label{eq:roots factorization}
L(u,\chi) = \prod_{i=1}^{d(\chi)}(1-\gamma_i u),
\end{equation}
we have $|\gamma_i| \in \{1, \sqrt{q}\}$ for all $1 \le i \le d(\chi)$.
\end{thm}

\begin{lem}\label{lem:sum von}
Let $\chi$ be a nonprincipal Dirichlet character, then
\begin{align}\label{eq:weil}
\bigg|\sum_{f \in \Mnq} \chi(f) \Lambda(f)\bigg| \le q^{\frac{n}{2}} d(\chi).
\end{align}
\end{lem}
\begin{proof}
We take the logarithmic derivative of the Euler product \eqref{eq:euler} and \eqref{eq:roots factorization} and compare coefficients to obtain 
\begin{equation}
\sum_{f \in \Mnq} \chi(f) \Lambda(f) = - \sum_{i=1}^{d(\chi)} \gamma_i^n
\end{equation}
for all $n \ge 1$. Then, Theorem \ref{thm:weil} implies the lemma via the triangle inequality.
\end{proof}

Let $\alpha\colon \Mq \to \CC$ be a completely multiplicative function. Given any subset $T$ of the positive integers $\NN$, we define
\[
F_{T}(u,\alpha) := \sum_{\substack{f \in \Mq \\ P \mid f \implies \deg P \in T }} \alpha(f)u^{\deg f}=
\prod_{P \in \Pq, \, \deg P \in T} (1-\alpha(P)u^{\deg P})^{-1}.
\]
Let us write $[u^m]F$ for the $m$th coefficient of a power series $F$. Our starting point is the identity
\begin{equation}\label{eq:iden}
\begin{split}
	\sum_{f \in \Mnq} \alpha(f) &= [u^n]F_{\NN}(u,\alpha)= [u^n]F_{T^c}(u,\alpha)F_{T}(u,\alpha) \\
	&= [u^n] F_{T^c}(u,\alpha) + [u^n] F_{T^c}(u,\alpha)(F_{T}(u,\alpha) -1)\\
 &=q^{-n}\sum_{\substack{f \in \Mnq\\ P \mid f \implies \deg P \not\in T}} \alpha(f)+q^{-n}\sum_{\substack{f \in \Mnq\\ \exists P \mid f \text{ such that }\deg P \in T}} \alpha(f),
 \end{split}
 \end{equation}
where $T^c = \NN \setminus T$ is the complement of $T$.

\subsection{Montgomery and Vaughan's estimate}
To bound the second term in \eqref{eq:iden}, we prove the following lemma generalizing an estimate of Montgomery and Vaughan \cite[Lem.~2]{Montgomery1977} in the polynomial setting.

\begin{lem}\label{lem:MV}
Let $\alpha\colon \Mq \to \CC$ be a $1$-bounded completely multiplicative function, meaning $|\alpha(f)|\le 1$ for all $f \in \Mq$. Let $n \ge 1$ and $S \subseteq \{1,\ldots,n\}$ be a set of positive integers, and $S^c = \{1,\ldots,n\} \setminus S$ be its complement. Let $s_0 = \min_{s \in S} s$. Then,
\[
 q^{-n}\sum_{\substack{f \in \Mnq\\ \exists P \mid f \text{ such that }\deg P \in S}} \alpha(f)\ll \exp(A_1)(\exp(A_2)-1),
\]
where
\[ A_1 = \sum_{d \in S^c} \frac{1}{d}, \qquad A_2 = \sum_{d \in S} \frac{q^{-d}}{d} \bigg| \sum_{f \in \mathcal{M}_{d,q}} \alpha(f)\Lambda(f)\bigg| + O(q^{-\frac{s_0}{2}}/s_0).\]
\end{lem}
\begin{remark}\label{rem:bho}
In \cite{Bhowmick2015} Lemma \ref{lem:MV} is proved in the special case $S =\{m+1,m+2,m+3,\ldots,n\}$ and $\alpha$ being a Dirichlet character.  For this choice, one immediately recovers \eqref{eq:bho3} from \eqref{eq:weil}.
\end{remark}
\begin{proof}
Using identity \eqref{eq:iden} and noticing that we can restrict the set of degrees to $\{1,\ldots,n\}$, we have
\[
X := q^{-n}\sum_{\substack{f \in \Mnq\\ \exists P \mid f \text{ such that }\deg P \in S}} \alpha(f) = [u^n] F_{S^c}(u,\alpha)(F_{S}(u,\alpha) -1).
\]
Because $\alpha$ is $1$-bounded, the coefficients in the series of $F_{S^c}(u,\alpha)$ are bounded in absolute value by the respective coefficients of 
\[ Z_{S^c}(u):=\sum_{\substack{f \in \Mq\\ P \mid f \implies \deg P \in S^c}} u^{\deg f} =\prod_{P \in \Pq,\, \deg P \in S^c} (1-u^{\deg P})^{-1}.\]
Further, we may write $F_S(u,\alpha)$ as
\begin{align}
F_{S}(u,\alpha) &= \prod_{P \in \Pq, \, \deg P \in S} (1-\alpha(P)u^{\deg P})^{-1} = \exp \left( \sum_{d \in S}\sum_{i \ge 1} i^{-1} u^{di}\sum_{P \in \mathcal{P}_{d,q}} \alpha(P)^i \right)\\
&=\exp \left( \sum_{d \in S} u^d \sum_{P \in \mathcal{P}_{d,q}}\alpha(P) + \sum_{d \in S} \sum_{i \ge 2}i^{-1}u^{di} \sum_{P \in \mathcal{P}_{d,q}} \alpha(P)^i\right)\\
&=\exp \left( \sum_{d \in S} \frac{u^d}{d}\sum_{f \in \mathcal{M}_{d,q}}\alpha(f)\Lambda(f) +\sum_{d \in S}\sum_{i \ge 2}\frac{u^{di}}{i}  \sum_{P \in \mathcal{P}_{d,q}} \alpha(P)^i-\sum_{d \in S}\frac{u^d}{d} \sum_{\substack{e \mid d\\ e \neq d}} e\sum_{P \in \mathcal{P}_{e,q}} \alpha(P)^{d/e} \right).
\end{align}
It follows that $[u^j](F_{S}(u,\alpha)-1)$ for $0\le j \le n$ are bounded in absolute value by the coefficients of $u^j$ in 
\begin{equation}\label{eq:Zsalphadef}
Z_{S,\alpha}(u) := \exp \left( \sum_{d \in S}\frac{u^d}{d} \bigg|\sum_{f \in \mathcal{M}_{d,q}} \alpha(f)\Lambda(f)\bigg|  +  \sum_{d \in S}\sum_{i \ge 2} \frac{u^{di}}{i} | \mathcal{P}_{d,q}| + \sum_{d \in S} \frac{u^d}{d} \sum_{\substack{e \mid d\\ e \neq d}} e |\mathcal{P}_{e,q}|\right)-1.
\end{equation}
Putting this together, we have
\[ 
|X|=\big|[u^n] F_{S^c}(u,\alpha)(F_{S}(u,\alpha) -1)\big| \le [u^n] Z_{S^c}(u) (Z_{S,\alpha}(u)-1).
\]
We make a general observation. If the coefficients of a power series $F(u)=\sum_{i \ge 0} f_i u^i$ are bounded in absolute value by the respective coefficients of $G(u)=\sum_{i \ge 0} g_i u^i$, then $|f_n| \le G(R)R^{-n}$ for every $0<R<C$ such that $G(R)$ converges. Indeed,
\[ G(R)R^{-n} = \sum_{i=0}^{\infty} g_i R^{i-n} \ge g_n R^{n-n} = g_n \ge |f_n|.\]
Applying this observation with $R=1/q$, we get
\begin{equation}\label{eq:nonneg}
|X|\le q^n Z_{S^c}(1/q) (Z_{S,\alpha}(1/q)-1).
\end{equation}
We estimate $Z_{S^c}(1/q)$ using Lemma \ref{lem:pnq} as
\begin{equation}\label{eq:zsctriv}
Z_{S^c}(1/q) = \exp\brackets{\sum_{\deg P\in S^c} (q^{-\deg P}+O(q^{-2\deg P}))} =\exp\brackets{ \sum_{d \in S^c}d^{-1} + O(1)}.
\end{equation}
Similarly, 
\begin{equation}\label{eq:B}
 \sum_{d \in S}\sum_{i \ge 2} \frac{q^{-di}}{i} | \mathcal{P}_{d,q}| + \sum_{d \in S} \frac{q^{-d}}{d} \sum_{\substack{e \mid d\\ e \neq d}} e |\mathcal{P}_{e,q}|\le \sum_{d \in S} \sum_{i \ge 2} \frac{q^{-di}}{i} \frac{q^d}{d} + 2\sum_{d \in S} \frac{q^{-\frac{d}{2}}}{d}\ll \frac{q^{-s_0/ 2}}{s_0},
\end{equation}
where $s_0 = \min_{s\in S} s$. The required estimate then follows from \eqref{eq:nonneg}, \eqref{eq:zsctriv}, \eqref{eq:Zsalphadef}, and \eqref{eq:B}.
\end{proof}
We now move on to estimating the first term in \eqref{eq:iden}.

\subsection{Sieve estimate}

\begin{lem}[Sieve bound]\label{lem:sieve}
	Let $S \subseteq \{1,\ldots,n\}$ be a set of positive integers, and $S^c = \{1,\ldots,n\}\setminus S$ be its complement. Then
	\[ \sum_{\substack{f \in \Mnq \\ P \mid f \implies \deg P \not\in S}} 1 \ll \frac{q^{n}}{n} \exp\brackets{\sum_{d \in S^c} d^{-1}} \asymp q^{n} \exp\brackets{- \sum_{d \in S} d^{-1}}.\]
\end{lem}

\begin{remark} In the integer setting, the estimate $\sum_{n\le x: \, (n,k)=1} 1 \ll x \prod_{p\mid k}(1-1/p)$ for any positive integer $k$ whose prime factors do not exceed $x$ is a classical consequence of Selberg's sieve (see \cite{Hall} for a discussion of this and an alternative proof). A permutation analogue of Lemma \ref{lem:sieve} was established (in greater generality) by Ford \cite[Thm.~1.5]{Ford}. The proof we give is self-contained and is in the spirit of \cite{Hall,Halberstam,Ford}.
\end{remark}

\begin{proof}
Let
\[ 
A_i := \sum_{\substack{f \in \mathcal{M}_{i,q} \\ P \mid f \implies \deg P \not\in S}}1,
\]
and define
\begin{equation}\label{eq:deff} 
F(u) :=  \prod_{\deg P \in S^c} (1-u^{\deg P})^{-1} = \exp\left( \sum_{\deg P \in S^c} \sum_{k \ge 1} \frac{u^{k\deg P }}{k}\right)=:\sum_{i=0}^{\infty} \widetilde{A}_i u^i.
\end{equation}
The coefficients satisfy $\widetilde{A}_i = A_i$ for $0 \le i \le n$.
Differentiating \eqref{eq:deff} formally, we see that
\begin{equation}\label{eq:diff}
\sum_{i \ge 1} i \widetilde{A}_i u^{i} = F(u)G(u),
\end{equation}
where
\[ 
G(u) = \sum_{i \ge 1} B_i u^i, \qquad  B_i := \sum_{\substack{k \ge 1,\,\deg P \in S^c \\ k \deg P = i}} \deg P = \sum_{\substack{d \in S^c \\ d \mid i}} d | \mathcal{P}_{d,q}| .
\]
Comparing coefficients in \eqref{eq:diff}, we obtain
\[ A_n = n^{-1} \sum_{i=1}^{n} A_{n-i} B_i.\]
Lemma \ref{lem:pnq} implies
\[ 
B_i \le \mathbf{1}_{i \not\in S} (i | \mathcal{P}_{i,q}|) + \sum_{d \le i/2} d| \mathcal{P}_{d,q}| \le \mathbf{1}_{i \not\in S} q^i +2q^{i/2},
\]
and so 
\[ 
A_n \le n^{-1} \sum_{i \in S^c} q^i A_{n-i} + 2n^{-1} \sum_{i=1}^{n} A_{n-i}q^{i/2}.
\]
The trivial bound $A_{n-i} \le q^{n-i}$ implies
\[ 
A_n \le n^{-1} \sum_{i \in S^c} q^i A_{n-i} + 5n^{-1} q^n \le n^{-1}q^n \brackets{ \sum_{0 \le i \le n} q^{-i} A_{i} + 5}.
\]
Since \[ F(1/q) =\sum_{i \ge 0} q^{-i}\widetilde{A}_i \ge \sum_{0\le i \le n} q^{-i} A_i,\] we find that
\[ A_n \le n^{-1} q^n \bigl( F(1/q) + 5\bigr). \]
By Lemma \ref{lem:pnq},
\[ 
\log F(1/q) = \sum_{d \in S^c} |\mathcal{P}_{d,q}| \sum_{k \ge 1}\frac{q^{-kd}}{k} \le \sum_{d \in S^c}\frac{1+q^{-d}}{d} \le \sum_{d \in S^c} \frac{1}{d} + 1,
\]
and recalling $\log n = \sum_{d=1}^{n} d^{-1} + O(1)$, we conclude that
\[ A_n \ll n^{-1} q^n \exp\brackets{ \sum_{d \in S^c} d^{-1}} \asymp q^n \exp\brackets{-\sum_{d \in S} d^{-1}}\]
as needed.
\end{proof}
\subsection{Proof of Lemma \ref{lem:char}}\label{sec:prooflemchar}
Due to the trivial estimate $|S(n,\chi)|\le q^n$, we may assume $n \ge 12(1+\log_q (1+\deg Q))$. Since $S(n,\chi)=0$ for $n\ge \deg Q$ we may also assume $\deg Q >n$. Applying Lemma \ref{lem:MV} with $\alpha=\chi$ and $S=\{m+1,m+2,\ldots,n\}$ for some $m \in (2\log_q \deg Q , n]$ (as indicated in Remark \ref{rem:bho}), we find that \eqref{eq:bho3} indeed holds by appealing to Lemma \ref{lem:sum von}. By Lemma \ref{lem:sieve},
\[q^{-n} \sum_{\substack{f \in \Mnq \\  P \mid f \implies \deg P \le m}} \chi(f) \ll q^{-n} \sum_{\substack{f \in \Mnq \\  P \mid f \implies \deg P \le m}} 1 \ll \frac{m}{n}.\]
Choosing $m= \lceil 2\log_q (n(1+\deg Q))\rceil $ in \eqref{eq:bho3} yields the statement of the lemma.

\subsection{Bound on a \texorpdfstring{$\gcd$}{gcd} sum}
For $L \ge 1$, let 
\[ B_{L,k} := \sum_{L\le d < 2L} \log \gcd(k,q^d-1).\]
We trivially have $B_{L,k} \ll L \min\{\log k, L\log q\}$, and this bound is optimal when $\log_q k \gg L^2$. For example, take $k=\prod_{1\le i < 2L}(q^i-1) = q^{\Theta(L^2)}$, then $q^d-1 \mid k$ for all $L\le d< 2L$ so that $B_{L,k} \asymp L^2 \log q$. For smaller $k$, however, this bound is too generous.

\begin{lem}\label{lem:Blk}
Suppose $L \ge \sqrt{\log_q k}$. Then $B_{L,k} \ll L \sqrt{\log k \log q}$.
\end{lem}

To prove this result we will need a couple of facts concerning cyclotomic polynomials. Recall that the cyclotomic polynomials $(\phi_n(x))_{n \ge 1}$ are defined recursively by
$\prod_{d \mid n} \phi_d(x) = x^n-1$.
They lie in $\ZZ[x]$ and can be written as $\phi_n(x) = \prod_{j \in (\ZZ/n\ZZ)^{\times}}(x-e^{2\pi i j/n})$. This last relation shows that if $m$ and $p$ are coprime where $p$ is a prime then $\phi_{mp^i}(x)=\phi_m(x^{p^i})/\phi_m(x^{p^{i-1}})$ for $i \ge 1$. 

The following lemma is classical but we could not find it explicitly stated in the literature and so we give a proof based on an argument implicit in Roitman \cite{Roitman}.

\begin{lem}\label{lem:cyc}
Fix an integer $a$. For a prime $p$ define  $A_{p} := \{ n \ge 1: p \mid \phi_n(a)\}$.
If $p\mid a$, then $A_p=\varnothing$. Otherwise, $A_p = \{p^i \ordpa : i \ge 0\}$ where $\ordpa$ is the multiplicative order of $a$ modulo $p$.
\end{lem}
\begin{proof}
If $p \mid \phi_n(a)$, then $p \mid a^n-1 $ since $\phi_n(a) \mid a^n-1$. If $p\mid a$, then $a^n -1 \equiv - 1 \bmod p$, which is a contradiction. Thus, $A_p=\varnothing$ in this case. 

Now suppose $p \nmid a$. By definition, we have $p \mid a^{\ordpa}-1$. Suppose $ p\mid \phi_n(a) \mid a^n -1$ for some positive integer $n$, then
\[ p \mid \gcd(a^n-1,a^{\ordpa}-1)= a^{\gcd(n,\ordpa)}-1.\]
Since $\ordpa = \min\{n \ge 1: p \mid a^n -1$\}, we have $\gcd(n,\ordpa) \ge \ordpa$ and thus $\ordpa \mid n$. This shows 
\begin{equation}\label{eq:inc} A_p \subseteq \{ m \cdot \ordpa: m \ge 1\}.
\end{equation}
Next we show $\ordpa \in A_p$. Since $p\mid a^{\ordpa}-1 = \prod_{e \mid \ordpa}\phi_e(a)$, it follows that $A_p$ contains some divisor of $\ordpa$, which by  \eqref{eq:inc} then has to be $\ordpa$ itself. Next, we want to show that for $n \in A_p$, $n/\ordpa$ is a power of $p$. If $p^\prime$ is a prime dividing $n/\ordpa$, then $\ordpa \mid n/p'$ and so
\[ 
p \mid \phi_{n}(a) = \frac{a^{n}-1}{\prod_{d \mid n,\, d \neq n} \phi_d(a)} \mid \frac{a^{n}-1}{\prod_{d \mid n/p'} \phi_d(a)} =\frac{a^{n}-1}{a^{n/p'}-1} = \sum_{i=0}^{p'-1} a^{in/p'} \equiv \sum_{i=0}^{p'-1} 1 = p' \bmod p,
\]
where we used that by definition of $\ordpa$
\[
a^{n/p'} = (a^{\ordpa})^{\frac{n}{p' \ordpa}} \equiv 1^{\frac{n}{p'\ordpa}} \equiv 1 \bmod p.
\]
Hence, $p^\prime$ must be $p$, and $n$ is indeed $\ordpa$ times a power of $p$. This shows $A_p \subseteq \{p^i \ordpa: i \ge 0\}$. Finally, the relation $\phi_{p^i m}(a) = \phi_{m}(a^{p^i})/\phi_m(a^{p^{i-1}})$ for $\gcd(m,p)=1$ and $i \ge 1$ implies $p^i \ordpa \in A_p$ for every $i \ge 1$ since $\phi_{m}(a^{p^i})/\phi_{m}(a^{p^{i-1}}) \equiv \phi_m(a)^{p^i-p^{i-1}} \bmod p$.
\end{proof}
 
\begin{proof}[Proof of Lemma \ref{lem:Blk}]
Recall that $q^d-1 = \prod_{e \mid d} \phi_e(q)$.
Because $\gcd(A,ab) \le \gcd(A,a)\gcd(A,b)$, it then follows that
\[ 
\log \gcd(k,q^d-1) \le \sum_{e \mid d} \log \gcd(k,\phi_e(q)) = \sum_{e \mid d} \sum_{p \mid k} \log \gcd(p^{\nu_p(k)},\phi_e(q)),
\]
where $p$ stands for prime and $\nu_p(k)$ is the multiplicity of $p$ in $k$ (i.e., the $p$-adic valuation of $k$). Hence,
\[ 
B_{L,k} \le \sum_{L\le d < 2L} \sum_{\substack{p \mid k,\,\phi_e(q)\\ e \mid d}} \log \gcd(p^{\nu_p(k)},\phi_e(q)).\]
We introduce a parameter $T\ge 1$. We shall show
\begin{equation}\label{eq:BLkbound}
     B_{L,k} \ll  T \log k+ \frac{L^2 \log q}{T},
\end{equation}
and then take $T=L/\sqrt{\log_q k}$.
We consider separately the contribution of $e> d/T$ and $e \le d/T$, obtaining
\[B_{L,k} \le B_{L,k,1}+B_{L,k,2},\]
where
\begin{align}
B_{L,k,1} &=\sum_{L\le d < 2L} \sum_{\substack{p \mid k,\,\phi_e(q)\\ e \mid d \\ e >d/T}} \log \gcd(p^{\nu_p(k)},\phi_e(q)) ,\\
B_{L,k,2} &=\sum_{L\le d < 2L} \sum_{\substack{p \mid k,\,\phi_e(q)\\ e \mid d \\ e \le d/T}} \log \gcd(p^{\nu_p(k)},\phi_e(q)).
\end{align}
Let us bound $B_{L,k,1}$. By interchanging the order of summation,
\begin{align}
B_{L,k,1} \le \sum_{\substack{e, \, p: \,p \mid k,\, \phi_e(q) \\ L/T < e < 2L}} \nu_p(k)\log p\sum_{\substack{L\le d < 2L\\ e \mid d}} 1 \ll 	L \sum_{\substack{e, \, p: \\p \mid k,\,\phi_e(q) \\ L/T <e < 2L}} \frac{\nu_p(k)\log p }{e} .
\end{align}
Let $A_p:=\{ n\ge 1: p \text{ divides }\phi_n(q)\}$. By Lemma \ref{lem:cyc}, $A_p$ is either empty or is a geometric progression with step size $p$, so that 
\[ \sum_{\substack{ e:\,  p\mid \phi_e(q)\\  L/T <e < 2L}} \frac{1}{e}  \ll \frac{T}{L} \sum_{i \ge 0} p^{-i} \ll \frac{T}{L}\]
and it follows that
\begin{align}
B_{L,k,1} &\ll T \sum_{p \mid k}  \nu_p(k) \log p = T \log k. 
\end{align} 
Next, omitting the restriction $p \mid k$ and using $\phi_e(q) \mid q^e-1 \le q^e$, we have
\[
B_{L,k,2} \le \sum_{L\le d < 2L} \sum_{\substack{e \mid d\\ e \le d/T}} \log \phi_e(q)\le \log q \sum_{L\le d < 2L} \sum_{\substack{e \mid d\\ e \le d/T}} e \ll  L\log q \sum_{1\le e < 2L/T} 1 \ll \frac{L^2 \log q}{T}.
\]
This implies \eqref{eq:BLkbound}, and thus the statement of the lemma.
\end{proof}

\subsection{Symmetry of character sums}
\begin{lem}\label{lem:sym}
Let $\psi\colon \FF_q \to \CC^{\times}$ be an additive character. Let $k,n \ge 1$, and let $k^\prime = \gcd(k,q^n-1)$. Then, 
\[ \sum_{f \in \Mnq} \Lambda(f)\chi_{k,\psi}(f) =\sum_{f \in \Mnq} \Lambda(f)\chi_{k^\prime,\psi}(f).\]
\end{lem}
\begin{proof}
Consider a surjective map \[\Phi\colon \FF_{q^{n}}^{\times} \to \{ f \in \Mnq \setminus \{ T\}: \Lambda(f) \neq 0\}\]
given by \[\Phi(x) = \prod_{\sigma \in \mathrm{Gal}(\FF_{q^n}/\FF_q)}(T-\sigma(x))=\prod_{i=0}^{n-1}(T-x^{q^i}). \]
Alternatively, one can also write $\Phi(x)=m_{\alpha}(x)^{n/\deg m_{\alpha}}$, where $m_{\alpha}$ is the minimal polynomial of $x$ over $\FF_q$. Every element $f$ in the image of $\Phi$ has $\Lambda(f)$ preimages given by the Galois conjugates of a given preimage, namely, $\{x^{q^i}: 0 \le i <d\}$ if $\deg m_{x}=d$. Because $\chi_{k,\psi}(f)$ is defined in terms of the zeros of $f$, the map $\Phi$ allows us to conveniently write
\[ \sum_{f \in \Mnq}\Lambda(f)\chi_{k,\psi}(f) = \sum_{x \in \FF_{q^n}^{\times}} \psi(\sum_{i=0}^{n-1} (x^{q^i})^{-k}).\]
By Euclid's algorithm, we can write $k^\prime = \gcd(k,q^n-1)=ak+b(q^n-1)$ for some coprime integers $a$ and $b$. In particular, $x^{k^\prime}=(x^k)^a$. Conversely, $x^k = (x^{k^\prime})^{k/k^\prime}$. It follows that the group endomorphisms $x \mapsto x^{k}$ and $x \mapsto x^{\gcd(k,q^n-1)}$ defined on $\FF_{q^n}^{\times}$ have the same image and kernel. This implies that every element in the image is attained the same number of times, and so 
\[\sum_{x \in \FF_{q^n}^{\times}} \psi(\sum_{i=0}^{n-1} (x^{q^i})^{-k})=\sum_{x \in \FF_{q^n}^{\times}} \psi(\sum_{i=0}^{n-1} (x^{q^i})^{-k^\prime})= \sum_{f \in \Mnq}\Lambda(f)\chi_{k^\prime,\psi}(f) \]
as required.
\end{proof}
From Lemmas \ref{lem:invchar}, \ref{lem:sum von}, and \ref{lem:sym}, we conclude the following.

\begin{cor}\label{cor:sym}
Let $\psi\colon \FF_q \to \CC^{\times}$ be a nontrivial additive character. Let $k,n \ge 1$. We have
\[ \bigg|\sum_{f \in \Mnq} \Lambda(f)\chi_{k,\psi}(f)\bigg| \le q^{\frac{n}{2}} \gcd(k,q^n-1).\]
\end{cor}

\subsection{Proof of Proposition \ref{prop:crit}}

Let $m:=\lceil 12\log_q n \rceil$, and 
\begin{equation}\label{eq:Sdef}
S=\{ m \le d \le n: \gcd(k,q^d-1) < q^{d/3}\}.
\end{equation}
Let $S^c = \{1,\ldots, n\} \setminus S$. As before, let
\[ q^{-n} \sum_{f \in \Mnq} \chi_{k,\psi}(f) =q^{-n} \sum_{\substack{f \in \Mnq \\ \exists P \mid f \text{ such that }\deg P \in S }}\chi_{k,\psi}(f) + q^{-n} \sum_{\substack{f \in \Mnq \\  P \mid f \implies \deg P \in S^c}} \chi_{k,\psi}(f).\]
Applying Lemma \ref{lem:MV} for $\alpha=\chi_{k,\psi}$ and our chosen $S$ together with Corollary \ref{cor:sym}, we get
\begin{align} 
&q^{-n} \sum_{\substack{f \in \Mnq \\ \exists P \mid f \text{ such that }\deg P \in S }}\chi_{k,\psi}(f) \\
&\ll \exp\left(\sum_{d \in S^c} d^{-1}\right) \left(\exp\left(\sum_{d \in S} \frac{q^{-d/6}}{d} + O(q^{-m/2}/m)\right)-1\right)\ll \frac{n q^{-m/6}}{m}\ll n^{-1}.
\end{align}
By Lemma \ref{lem:sieve},
\[q^{-n} \sum_{\substack{f \in \Mnq \\  P \mid f \implies \deg P \in S^c}} \chi_{k,\psi}(f) \ll q^{-n} \sum_{\substack{f \in \Mnq \\  P \mid f \implies \deg P  \in S^c}} 1 \ll \exp(-\sum_{d \in S}d^{-1}).\]
\begin{remark}\label{rem:mo}
The same proof works as is if we replace $\chi_{k,\psi}$ by $\mu\cdot \chi_{k,\psi}$, because
\[ \sum_{f \in \mathcal{M}_{d,q}} \chi_{k,\psi}(f)\Lambda(f) = -\sum_{f \in \mathcal{M}_{d,q}}\mu(f) \chi_{k,\psi}(f)\Lambda(f) + O(q^{d/2}).\]
\end{remark}
\subsection{Conclusion of proof of Theorem \ref{thm:canc}}
Due to the trivial estimate $|S(n,\chi)|\le q^n$, we may assume $n \ge 12(1+ \sqrt{\log_q 
k})$. In view of Proposition \ref{prop:crit} it suffices to show
\[ -\sum_{d \in S}d^{-1} \le \log (1+\sqrt{\log_q k} +\log_q n ) -\log n+O(1),\]
where $S$ is as in \eqref{eq:Sdef}. To show this, first observe that $-\sum_{d \in S} d^{-1} =\sum_{d \in S^c} d^{-1} - \log n + O(1)$. Let $m' = \max\{\lceil 12\log_q n\rceil, \sqrt{\log_q k}\}$. Since $\mathbf{1}_{A \ge B} \le \log A/\log B$, 
\[  \sum_{d \in S^c} d^{-1} \le \sum_{d \le m'} d^{-1} +  \sum_{\substack{m'\le d \le n \\ \gcd(k,q^d-1) \ge q^{d/3}}}  d^{-1} \le \log m' 
 + O(1) +\sum_{m'\le d \le n}d^{-1}\frac{\log \gcd(k,q^d-1)}{\log (q^{d/3})} . \]
By Lemma \ref{lem:Blk}, the last $d$-sum is $O(\sqrt{\log_q k}/m^\prime)$, finishing the proof.
\section*{Acknowledgments}
We thank Brad Rodgers and Zeev Rudnick for comments on an earlier version of the paper. We appreciate the comments and suggestions given by the anonymous referee, which improved the presentation of the paper. O.G. was supported by the European
Research Council (ERC) under the European Union's Horizon 2020 research and innovation programme
(grant agreement number 851318). V.K. was supported by ERC LogCorRM (grant number 740900) and by CRM-ISM postdoctoral fellowship. 

\appendix
\section{Squareroot cancellation in character sums}\label{sec:app}
Under the generalized Riemann hypothesis for $L(s,\chi)$ one has $\sum_{n \le x} \chi(n) \ll \sqrt{x} \exp(C\log m/\log \log m)$ for a nonprincipal Dirichlet character $\chi$ modulo $m$ \cite{Bhowmick2017}.
Bhowmick, L\^{e} and Liu \cite{Bhowmick2017} proved the function field analogue
\begin{equation}\label{eq:bho}
	\sum_{f \in \Mnq} \chi(f) \ll q^{n/2} \exp\left( C_q\left(\frac{ n \log \log \deg Q}{\log \deg Q}+ \frac{\deg Q}{\log^2 \deg Q}\right)\right)
\end{equation}
unconditionally, for a constant $C_q$ that depends only on $q$, where $\chi$ is a character modulo $Q \in \FF_q[T]$. The estimate \eqref{eq:bho} implies
$\sum_{f \in \Mnq}\chi(f) \ll q^{\frac{n}{2}(1+o(1))}$ for $\deg Q=o(n\log^2 n)$. 
Here we prove that squareroot cancellation holds in the wider range $\deg Q\le n^{1+o(1)}$.
\begin{prop}\label{prop:squareroot}
	Let $\chi$ be a nonprincipal Dirichlet character modulo $Q \in \FF_q[T]$. There is a constant $C_q$ that depends only on $q$ such that 
	\begin{equation}
		\sum_{f \in \Mnq} \chi(f) \ll_q q^{\frac{n}{2}} \exp\left(  \max\left\{C_q n\frac{\log ( \deg Q\log \deg Q/n)}{\log (1+\deg Q)},0\right\}  \right).
	\end{equation}
\end{prop}
Recall $d(\chi):=\deg L(u,\chi)<\deg Q$.
If $d(\chi) <n$ we have $S(n,\chi)=0$, so from now on we assume that $d(\chi) \ge n$. In view of the trivial bound $|S(n,\chi)| \le q^n$ we may also assume that $n$ is larger than a fixed constant, as well as that $d \le n^{1+\delta}$ for a fixed $\delta>0$. From \eqref{eq:euler},
\begin{equation}\label{eq:lgl log}
	L(u,\chi) =\exp\left( \sum_{k \ge 1} \frac{u^k}{k} \sum_{f \in \mathcal{M}_{k,q}} \chi(f) \Lambda(f) \right).
\end{equation}
Set $L:=\lfloor 2\log_q d \rfloor$. For $k \le L$ the trivial bound 
\begin{equation}\label{eq:trivial}
	\left| \sum_{f \in \mathcal{M}_{k,q}} \chi(f) \Lambda(f)\right| \le \sum_{f \in \mathcal{M}_{k,q}}\Lambda(f) = q^k
\end{equation}
coming from \eqref{eq:gauss} is superior to \eqref{eq:weil}. We see from \eqref{eq:weil} and \eqref{eq:trivial} that the coefficients of $L(u,\chi)$ are bounded (in absolute value) from above by those of \[\exp\bigg( \sum_{k \ge 1} \frac{u^k}{k} \min\{q^k, d(\chi) q^{k/2}\} \bigg),\] hence, writing $[u^n]F$ for the $n$th coefficient of a power series $F$,
\begin{align}
	|S(n,\chi)| = q^{n/2}|[u^n]L(u/\sqrt{q},\chi)| &\le q^{n/2}[u^n] \exp\bigg( \sum_{k \ge 1} \frac{u^k}{k} \min\{q^{k/2}, d(\chi)\} \bigg) \\
	&\le q^{n/2} \exp\bigg( \sum_{k\ge 1} \frac{R^k}{k} \min\{q^{k/2}, d(\chi)\} \bigg) R^{-n}
\end{align}
for any $R \in (0,1)$ (this is the same observation used in the proof of Lemma \ref{lem:MV}). If	$6/(5\sqrt{q}) < R < 1$ then
\begin{equation}
	\sum_{1 \le k \le L} \frac{R^{k}}{k} \min\{q^{k/2},d(\chi)\} = \sum_{1 \le k \le L} \frac{(R\sqrt{q})^k}{k} \le  \sum_{1 \le k \le L} (R\sqrt{q})^k  \le \frac{(R\sqrt{q})^L}{1-(R\sqrt{q})^{-1}} \le 6d(\chi) R^L 
\end{equation}
and 
\begin{equation}
	\sum_{k > L} \frac{R^k}{k} \min\{q^{k/2}, d(\chi)\}  \le d(\chi) \sum_{k>L} \frac{R^k}{k} \le \frac{d(\chi)}{L+1} \sum_{k>L} R^k = \frac{d(\chi)}{L+1}  \frac{R^{L+1}}{1-R},
\end{equation}
so that
\begin{equation}\label{eq:zurbound}
	|S(n,\chi)/q^{n/2}| \le \exp\Big( 6d(\chi)R^L \Big( 1 + \frac{R}{(L+1)(1-R)}\Big)  - n \log R\Big).    
\end{equation}
Set $a := \log (d(\chi) \log d(\chi) /n)$ and take $R=q^{-a/(2\log d(\chi))}$. The assumption $d(\chi) \ge n$ implies $a > 0$ (at least if $n \ge 3$) and so $R < 1$. (This choice of $R$ differs from the choice in \cite{Bhowmick2017}, where $a$ is chosen to be of order $\log \log d(\chi)$.) Since we may assume $n \le d(\chi) \le n^{1+\delta}$ and $n \ge C$ for $\delta$ and $C$ we choose, we may also assume $a/\log d(\chi)$ is at most $\delta'$ for a $\delta'$ we want. In particular, $R > 6/(5\sqrt{q})$ may be assumed from now on by taking small enough $\delta'$. Moreover,	\begin{equation}\label{eq:valsinexp}
	-n \log R \le  (\log q) n \frac{\log (d(\chi) \log d(\chi) /n)}{\log d(\chi)}, \qquad R^L \le   q^{-\frac{\log (d(\chi) \log d(\chi)/n)}{2\log d(\chi)} (2 \log_q d(\chi) - 1)} \le \frac{2}{(d(\chi) \log d(\chi)/n)}
\end{equation}
for large enough $n$. Additionally, $a\log q/\log d(\chi) \in (0,1)$ holds (by taking small enough $\delta'$) which implies $R \le 1-a\log q/(4\log d(\chi))$, and so
\begin{equation}\label{eq:valsinexp2}
	\frac{R}{(L+1)(1-R)} \le \frac{1}{L+1} \frac{1}{1-R} \le \frac{\log q}{2 \log d(\chi)} \frac{ 4\log d(\chi)}{\log q \log (d(\chi) \log d(\chi) /n)} = \frac{2}{\log (d(\chi) \log d(\chi)/n)} \le 1
\end{equation}
for large enough $n$. All in all,
\begin{equation}\label{eq:zurbound2}
	|S(n,\chi)/q^{n/2}| \le \exp\bigg( \frac{24n}{\log d(\chi)} +(\log q) n \frac{\log (d(\chi) \log d(\chi)/n)}{\log d(\chi)} \bigg) \le \exp\bigg( C_q n \frac{\log (d(\chi) \log d(\chi)/n)}{\log d(\chi)} \bigg).
\end{equation}
Since $\deg Q > d(\chi)$, the claim follows.

\bibliographystyle{abbrv}
\bibliography{references} 
\Addresses
\end{document}